\documentclass[a4paper,11pt,reqno]{article}

\usepackage{amsmath}
\usepackage{amssymb}
\usepackage{amsopn}
\usepackage{amsthm}
\usepackage{amstext}
\newtheorem{proposition}{Proposition}[section]
\newtheorem{theorem}[proposition]{Theorem}
\newtheorem{lemma}[proposition]{Lemma}
\newtheorem{corollary}[proposition]{Corollary}
\newtheorem{definition}[proposition]{Definition}
\newtheorem{remark}[proposition]{Remark}

\begin{document}
\date{}
\title{The characterizations of the stable perturbation of a closed operator by a linear operator in Banach spaces}
\author{Fapeng Du\thanks{E-mail: jsdfp@163.com}\\
   School of Mathematical \& Physical Sciences, Xuzhou Institute of Technology\\
  Xuzhou 221008, Jiangsu Province, P.R. China\\
 \and
 Yifeng Xue\thanks{Corresponding author, E-mail: yfxue@math.ecnu.edu.cn} \\
 Department of mathematics, East China Normal University\\
  Shanghai 200241, P.R. China
}

\maketitle

\begin{abstract}
In this paper, we investigate the invertibility of $I_Y+\delta TT^+$
when $T$ is a closed operator from $X$ to $Y$ with a generalized
inverse $T^+$ and $\delta T$ is a linear operator whose domain
contains $D(T)$ and range is contained in $D(T^+)$. The
characterizations of the stable perturbation $T+\delta T$ of $T$ by
$\delta T$ in Banach spaces are obtained. The results extend the recent
main results of Huang's in Linear Algebra and its Applications.

\vspace{3mm}

\end{abstract}

\setlength{\baselineskip}{15pt}

\section{Introduction}

The expression and perturbation analysis of the generalized inverse
(resp. the Moore--Penrose) inverse of bounded linear operators on
Banach spaces (resp. Hilbert spaces) have been widely studied since
Nashed's book \cite{Na} was published in 1976. Ten years ago, Chen
and Xue proposed a notation so--called the stable perturbation of a
bounded operator instead of the rank--preserving perturbation of a
matrix in \cite{CX1}. Using this new notation, they established the
perturbation analyses for the Moore--Penrose inverse and the least
square problem on Hilbert spaces in \cite{CWX}, \cite{CX2},
\cite{DJ} and \cite{XC}. In recent years, the perturbation analysis
of generalized inverses of closed operators has been appeared in
\cite{QHL}, \cite{QH} and \cite{WZ} with small perturbation
operators bounded related to closed operators. The results in these
papers generalize corresponding results in \cite{CX1}.

Throughout the paper, $X$ and $Y$ are always Banach spaces. Let
$B(X,Y)$, $D(X,Y)$ and $C(X,Y)$ denote the set of bounded linear
operators, densely--defined linear operators from $X$ to $Y$ and
closed densely--defined linear operators from $X$ to $Y$,
respectively. For $T \in D(X,Y)$, let $R(T)$ (resp. $N(T)$) denote
the range (resp. null space) of $T$. Suppose that $T\in C(X,Y)$ has
a generalized inverse $T^+$. Let $\delta T\colon D(\delta
T)\rightarrow Y$ be a closed operator with $D(T)\subset D(\delta T)$
and $R(\delta T)\subset D(T^+)$. Put $\bar T=T+\delta T$. In this
paper,  we first characterize when $I_Y+\delta TT^+ \colon
D(T^+)\rightarrow D(T^+)$ is bijective and then give some equivalent
conditions that make $R(\bar T)\cap N(T^+)=\{0\}$ under the
assumption that $I_Y+\delta TT^+\colon D(T^+)\rightarrow D(T^+)$ is
bijective. These results generalize several main results in
\cite{QHL,QH}.

\section{Some Lemmas}

Let $V$ be a closed subspace of $X$. Recall that $V$ is complemented
in $X$ if there is a closed subspace $U$ in $X$ such such $V\cap
U=\{0\}$ and $X=V+U$. In this case, we set $X= V\dotplus U$ and
$U=V^c$.

Let $T\in B(X,Y)$. If there is $S\in B(Y,X)$ such that $TST=T$ and
$STS=S$, then we say $T$ has a generalized inverse $S$, denoted by
$T^+$. It is well--known that $T\in B(X,Y)$ has a $T^+\in B(Y,X)$
iff $R(T)$ is closed and
$$
X=N(T)\dotplus N(T)^c,\quad Y=R(T)\dotplus R(T)^c
$$
(cf. \cite{CYX}). In general, we have
\begin{definition}
Let $T\in C(X,Y)$. If there is $S\in D(Y,X)$ with $D(S)\supset R(T)$
and $R(S)\subset D(T)$ such that
\begin{equation}\label{2eqa}
TST=T\ \text{on}\ D(T),\quad STS=S\ \text{on}\ D(S),
\end{equation}
then $S$ is called a generalized inverse of $T$, denoted by $T^+$.
\end{definition}

From (\ref{2eqa}), we get that $P=I_X-ST$ (resp. $Q=TS$) is an
idempotent operator on $D(T)$ (resp. $D(S)$) with $R(P)=N(T)$ (resp.
$R(Q)=R(T)$). Let $T\in C(X,Y)$. It is known that for $T\in C(X,Y)$,
we can always find a $T^+\in D(Y,X)$ (cf. \cite{Na}) and we call
$T^+$ is an algebraic generalized inverse of $T$. But when $T^+$
becomes a closed operator is a problem. The following proposition
(cf. \cite{Na}) gives an answer.

\begin{proposition}\label{2P1}
Let $T\in C(X,Y)$. Assume that
$Y=\overline{R(T)}\dotplus(\overline{R(T)})^c$. Let $Q\colon
Y\rightarrow\overline{R(T)}$ be the bounded idempotent operator on
$Y$.
\begin{enumerate}
\item[$(1)$] If there is a closed subspace $M$ of $X$ such that $M\cap N(T)=\{0\}$ and $D(T)=N(T)+M\cap D(T)$,
then $T^+\in C(Y,X)$ with $D(T^+)=R(T)+(\overline{R(T)})^c$,
$R(T^+)=D(T)\cap M$ and $TT^+y=Qy$, $\forall\,y\in D(T^+)$.
\item[$(2)$] If $X=N(T)\dotplus N(T)^c$, then there exists a unique $S\in C(Y,X)$ with $D(S)=R(T)+(\overline{R(T)})^c$,
$N(S)=(\overline{R(T)})^c$ and $R(S)=D(T)\cap N(T)^c$ such that
\begin{alignat}{3}
\label{2eqb} TST&=T\ \text{on}\ D(T)\,\ \text{and}&\,\ STS&=S\ \text{on}\ D(S)\\
\label{2eqc} TS&=Q\ \text{on}\ D(S)\,\ \text{and}&\,\ ST&=I_X-P\
\text{on}\ D(T),
\end{alignat}
where $P$ is the idempotent operator of $X$ onto $N(T)$.

In addition, $S$ is bounded if $R(T)$ is closed.
\end{enumerate}
\end{proposition}
\begin{proof}(1) Put $A=T\big\vert_{M\cap D(T)}$. It is easy to check that $A$ is a closed operator with $N(A)=\{0\}$
and $R(A)=R(T)$. Thus, $A^{-1}\colon R(T)\rightarrow M\cap D(T)$ is
also a closed operator. Set $Sy=\begin{cases}\, A^{-1}y\quad& y\in
R(T)\\ \, 0\quad& y\in (\overline{R(T)})^c\end{cases}$. Then
$D(S)=R(T)+(\overline{R(T)})^c$ is dense in $Y$, $R(S)=M\cap
D(T)\subset D(T)$ and
$$
TST=T\ \text{on}\ D(T),\quad STS=S\ \text{on}\ D(S),\ TS=Q\
\text{on}\ D(S).
$$

To show that $S\in C(Y,X)$, let $\{y_n\}^\infty_{n=1}\subset D(S)$
such that $\|y_n-y\|\to 0$ and $\|Sy_n-x\|\to 0$ as $n\to\infty$ for
some $y\in Y$ and $x\in X$. Note that $Qy_n\in R(T)$,
$Sy_n=SQy_n=A^{-1}Qy_n$, $n\ge 1$ and $\|Qy_n-Qy\|\to 0$. Since
$A^{-1}\in C(\overline{R(T)},X)$, it follows that $Qy\in R(T)$ and
$A^{-1}Qy=x$  and consequently, $y=Qy+(I_Y-Q)y\in D(S)$ and
$Sy=SQy=x$. Thus, $S\in C(Y,X)$.

(2) Let $M=N(T)^c$ in (1). Then by the proof of (1), $S$ satisfies
the requirements of Proposition \ref{2P1} (2).

Assume that there is another $S'\in C(Y,X)$ with
$D(S')=R(T)+(\overline{R(T)})^c$ such that $S'$ satisfies
(\ref{2eqb}) and (\ref{2eqc}). Then
$$
S'=S'TS'=(I_X-P)S'=STS'=SQ=STS=S\,\ \text{on}\ D(S).
$$

When $R(T)$ is closed, $D(S)=Y$. So $S$ is bounded by Closed Graph
Theorem.
\end{proof}

The operator $S$ in Proposition \ref{2P1} (2) is denoted by
$T^+_{P,\,Q}$.

Let $H$ and $K$ be Hilbert spaces. For a closed subspace $M$ in $H$
(or $K$), let $P_M$ denote the orthogonal projection from $H$ (or
$K$) to $M$. According to Proposition \ref{2P1} and its proof, we
have
\begin{corollary}\label{2C1}
Let $T\in C(H,K)$. Then there is a unique $S\in C(K,H)$ with
$D(S)=R(T)+R(T)^\perp$ and $R(S)=N(T)^\perp\cap D(T)$ such that
\begin{alignat*}{2}
TST&=T\ \text{on}\ D(T)\,\ &\text{and}\,\ STS&=S\ \text{on}\ D(S)\\
TSy&=P_{\overline{R(T)}}y,\ \forall\,y\in D(S)\,\ &\text{and}\,\
STx&=P_{\overline{N(T)^\perp\cap D(T)}}x,\ \forall\,x\in D(T).
\end{alignat*}

In addition, if $R(T)$ is closed, then $S$ is bounded.
\end{corollary}

The operator $S$ in Corollary \ref{2C1} is called the the maximal
Tseng inverse of $T$ (cf. \cite{BG}), denote by $T^\dag$. Clearly,
$N(T^\dag)=R(T)^\perp$ and $R(T^\dag)=N(T)^\perp\cap D(T)$.

\begin{lemma}\label{YL1}
Let $T\in C(X,Y)$ with $T^+\in D(Y,X)$ and let $\delta T\colon
D(\delta T)\subset X\rightarrow Y$ be a linear operator with
$D(T)\subset D(\delta T)$. Put $\bar{T}=T+\delta T$. If $R(\delta
T)\subset D(T^+)$, then $I_Y+\delta T T^+\colon D(T^+)\rightarrow
D(T^+)$ is bijective if and only if $I_X+T^+\delta T\colon
D(T)\rightarrow D(T)$ is bijective.
\end{lemma}
\begin{proof}
Suppose that $I_Y+\delta T T^+$ is bijective. Then there is an
operator $C\colon D(T^+)\rightarrow D(T^+)$ such that $(I_Y+\delta T
T^+)C=C(I_Y+\delta T T^+)=I_Y\ \text{on}\ D(T)$, that is,
\begin{equation}\label{2eqe}
C\delta T T^+=\delta TT^+C=I_X-C\ \text{on}\ D(T^+).
\end{equation}
Thus, from (\ref{2eqe}), we get that for any $\xi\in D(T^+)$,
\begin{align*}
(I_X+T^+\delta T)(I_X-T^+C\delta T)\xi&=\xi+T^+\delta T\xi-T^+C\delta T\xi-T^+\delta TT^+C\delta T\xi\\
&=\xi+T^+\delta T\xi-T^+C\delta T\xi-T^+(I_X-C)\delta T\xi\\
&=\xi.
\end{align*}
Similarly, $(I_X-T^+C\delta T)(I_X+T^+\delta T)\xi=\xi$,
$\forall\,\xi\in D(T^+)$. Therefore, $I_X+T^+\delta T$ is bijective.

Conversely, if $I_X+T^+\delta T$ is bijective, we can obtain that
$I_Y+\delta TT^+$ by using similar way.
\end{proof}

\begin{lemma}\label{YL2}
Let $T\in C(X,Y)$ with $T^+\in D(Y,X)$. Let $\delta T\colon D(\delta
T)\subset X\rightarrow D(T^+)$ be a linear operator such that
$D(T)\subset D(\delta T)$. Put $\bar T=T+\delta T$. Assume that
$I_X+T^+\delta T$ $\colon D(T)\rightarrow D(T)$ is bijective and
$R(\bar T)\cap N(T^+)=\{0\}$. Then $N(\bar T)=(I_X+T^+\delta
T)^{-1}N(T)$.
\end{lemma}
\begin{proof} Let $x\in N(\bar T)$. Then $Tx=-\delta Tx$ and $(I_X-T^+T)x=(I_X+T^+\delta T)x$. Note that $(I_X-T^+T)x\in
N(T)$ and $I_X+T^+\delta T$ is bijective. So $x\in (I_X+T^+\delta
T)^{-1}N(T)$.

Now let $x\in N(T)$ and put $z=(I_X+T^+\delta T)^{-1}x$. Then
$(I_X+T^+\delta T)z=x$ and $T(I_X+T^+\delta T)z=0$. Thus, $T^+\bar
Tz=0$. Since $R(\bar T)\cap N(T^+)=\{0\}$, it follows that $\bar
Tz=0$, i.e., $x\in N(\bar T)$.
\end{proof}

\section{Stable perturbation in Banach spaces}
\setcounter{equation}{0}

Let $T\in C(X,Y)$ and let $\delta T\colon D(\delta T)\rightarrow Y$
be a linear operator with $D(T)\subset D(\delta T)$. Recall that
$\delta T$ is $T$--bounded if there are constants $a,\,b>0$ such
that
$$
\|\delta Tx\|\le a\|x\|+b\|Tx\|,\quad \forall\,x\in D(T).
$$
We have known from \cite[Chap 4, Theorem 1.1]{TK} that $\bar
T=T+\delta T\in C(X,Y)$ when $\delta T$ is $T$--bounded with $b<1$.

Let $T\in C(X,Y)$ such that $T^+$ exists and let $\delta T\colon
D(\delta T)\rightarrow Y$ be a linear operator with $D(T)\subset
D(\delta T)$, $T$--bounded and $b<1$. Put $\bar T=T+\delta T\in
C(X,Y)$. According to \cite{CX1}, we say $\bar T$ is a stable
perturbation of $T$ if $R(\bar T)\cap N(T^+)=\{0\}$.

The following theorem characterizes when $I_Y+\delta TT^+\colon
D(T^+)\rightarrow D(T^+)$ is bijective and $\bar T$ is a stable
perturbation of $T$.
\begin{theorem}\label{DL1}
Let $T\in C(X,Y)$ with $T^+\in D(Y,X)$ and let $\delta T\colon
D(\delta T)\rightarrow D(T^+)$ be a linear operator such that
$D(T)\subset D(\delta T)$ and $\delta T$ is $T$--bounded with $b<1$.
Put $\bar T=T+\delta T\in C(X,Y)$. Then the following statements are
equivalent:
\begin{enumerate}
\item[$(1)$] $I_Y+\delta TT^+\colon D(T^+)\rightarrow D(T^+)$ is bijective\,$;$
\item[$(2)$] $T^+\bar T\big\vert_{R(T^+)}=(I_X+T^+\delta T)\vert_{R(T^+)}\colon R(T^+)\rightarrow R(T^+)$ is bijective\,$;$
\item[$(3)$] $D(T^+)=\bar TR(T^+)+N(T^+)$, $\bar TR(T^+)\cap N(T^+)=\{0\}$ and $N(\bar T)\cap R(T^+)=\{0\}$.

\end{enumerate}
\end{theorem}
\begin{proof} (1)$\Rightarrow$(2) Assume that $W=I_Y+\delta TT^+\colon D(T^+)\rightarrow D(T^+)$ is bijective. From $W=\bar TT^++(I_Y-TT^+)$
and $(I_Y-TT^+)D(T^+)=N(T^+)$, we get that
$D(T^+)=WD(T^+)\subset\bar TR(T^+)+N(T^+)$. Note that $\bar
TR(T^+)\subset D(T^+)$ and $N(T^+)\subset D(T^+)$. So $\bar
TR(T^+)+N(T^+)=D(T^+)$ and consequently, $R(T^+)=T^+\bar TR(T^+)$.
This shows that $D=T^+\bar T\big\vert_{R(T^+)}\colon
R(T^+)\rightarrow R(T^+)$ is surjective.

Now let $\xi\in R(T^+)$ and $T^+\bar T\xi=0$. Then
$$
(I_X+T^+\delta T)\xi=(I_X-T^+T)\xi+T^+\bar T\xi=0
$$
and consequently, $\xi=0$ by Lemma \ref{YL1}, that is, $D$ is
injective.

Noting that $T^+\bar TT^+=T^+(T+\delta T)T^+=(I_X+T^+\delta T)T^+$,
we have $D=(I_X+T^+\delta T)\big\vert_{R(T^+)}$.

(2)$\Rightarrow$(3) For any $\xi\in D(T^+)$ there is $\eta\in
D(T^+)$ such that $T^+\xi=T^+\bar TT^+\eta$ since $D$ is surjective.
Thus, $\zeta=\xi-\bar TT^+\eta\in N(T^+)$ and so that
$D(T^+)\subset\bar TR(T^+)+N(T^+)\subset D(T^+)$.

Let $\xi\in\bar TR(T^+)\cap N(T^+)$. Then $T^+\xi=0$ and $\xi=\bar
TT^+\eta$ for some $\eta\in D(T^+)$. So $DT^+\eta=0$. Since $D$ is
injective, we have $T^+\eta=0$ and so that $\xi=0$. This proves that
$\bar TR(T^+)\cap N(T^+)=\{0\}$.

Similarly, we can obtain $N(\bar T)\cap R(T^+)=\{0\}$.

(3)$\Rightarrow$(1) Since $D(T^+)=\bar TR(T^+)+N(T^+)$, it follows
that for any $\eta\in D(T^+)$, there is $\xi_1\in D(T^+)$ and
$\xi_2\in N(T^+)$ such that $\eta=\bar TT^+\xi_1+\xi_2$. Put
$\xi=TT^+\xi_1+\xi_2\in D(T^+)$. Then
$$
(I_Y+\delta TT^+)\xi=(I_Y-TT^+)\xi+\bar TT^+\xi=\xi_2+\bar
TT^+\xi_1=\eta,
$$
that is, $I_Y+\delta TT^+\colon D(T^+)\rightarrow D(T^+)$ is
surjective.

To prove $I_Y+\delta TT^+$ is injective, let $\zeta\in D(T^+)$ such
that $(I_Y+\delta TT^+)\zeta=0$. Then $(I_Y-TT^+)\zeta=-\bar
TT^+\zeta$. Since $\bar TR(T^+)\cap N(T^+)=\{0\}$, we get that
$TT^+\zeta=\zeta$ and $\bar TT^+\zeta=0$ and so $T^+\zeta\in N(\bar
T)\cap R(T^+)$. Now from the assumption that $N(\bar T)\cap
R(T^+)=\{0\}$, we obtain that $T^+\zeta=0$. Thus,
$\zeta=TT^+\zeta=0$.
\end{proof}

\begin{corollary}\label{3C1}
Let $T\in C(X,Y)$ with $T^+\in D(Y,X)$ and let $\delta T\colon
D(\delta T)\rightarrow D(T^+)$ be a linear operator such that
$D(T)\subset D(\delta T)$ and $\delta T$ is $T$--bounded with $b<1$.
Put $\bar T=T+\delta T\in C(X,Y)$.
\begin{enumerate}
\item[$(1)$] If $\bar T$ and $T$ satisfy following conditions:
\begin{alignat*}{2}
N(\bar{T})&\cap R(T^+)=\{0\},&\quad R(\bar{T})&\cap N(T^+)=\{0\},\\
D(T)&=N(\bar T)+R(T^+),&\quad D(T^+)&=N(T^+)+R(\bar T),
\end{alignat*}
then $I_Y+\delta TT^+\colon D(T^+)\rightarrow D(T^+)$ is bijective.
\item[$(2)$] If $I_Y+\delta TT^+\colon D(T^+)\rightarrow D(T^+)$ is bijective and $R(\bar T)\cap N(T^+)=\{0\}$, then
$D(T)=N(\bar T)+R(T^+)$ and $D(T^+)=N(T^+)+R(\bar T)$.
\end{enumerate}
\end{corollary}
\begin{proof} (1) $R(\bar T)\cap N(T^+)=\{0\}$ implies that $\bar TR(T^+)\cap N(T^+)=\{0\}$. Since
$D(T)=N(\bar T)+R(T^+)$, we have $R(\bar T)=\bar TR(T^+)$. Thus,
$$
D(T^+)=R(\bar T)+N(T^+)=\bar TR(T^+)+N(T^+)
$$
and hence $I_Y+\delta TT^+\colon D(T^+)\rightarrow D(T^+)$ is
bijective by Theorem \ref{DL1}.

(2) By Theorem \ref{DL1}, $D(T^+)=\bar TR(T^+)+N(T^+)$ when
$I_Y+\delta TT^+$ is bijective. Noting that $\bar TR(T^+)\subset
R(\bar T)\subset D(T^+)$,  we have $D(T^+)=N(T)+R(\bar T)$.

Since $I_X+T^+\delta T=I_X-T^+T+T^+\bar T$ is bijective by Lemma
\ref{YL1} and $(I_X+T^+\delta T)T^+=T^+(I_Y+\delta TT^+)$ on
$D(T^+)$, we have
$$
I_X=(I_X+T^+\delta T)^{-1}(I_X-T^+T)+T^+(I_Y+\delta TT^+)^{-1}\bar
T\quad \text{on}\ D(T).
$$
Therefore, $D(T)=N(\bar T)+R(T^+)$ by Lemma \ref{YL2}.
\end{proof}

Now we present our main result of the paper as follows.
\begin{theorem}\label{DL2}
Let $X,\,Y$ be Banach Spaces and let $T\in C(X,Y)$ with $T^+\in
D(Y,X)$. Let $\delta T\colon D(\delta T)\rightarrow D(T^+)$ be a
linear operator such that $D(T)\subset D(\delta T)$. Assume that
$\delta T$ is $T$--bounded with $b<1$ and $I_Y+\delta TT^+\colon
D(T^+)\rightarrow D(T^+)$ is bijective. Put $\bar{T}=T+\delta T$ and
$G=T^+(I_Y+\delta TT^+)^{-1}$. Consider following two statements
\rm{(A)} and \rm{(B)}. We have \vspace{1mm}

\noindent{\rm(A)} The following conditions are equivalent:
\begin{enumerate}
  \item[$(1)$] $R(\bar{T})\cap N(T^+)=\{0\};$
  \item[$(2)$] $G=\bar T^+\in D(Y,X)$ with $R(G)=R(T^+)$, $N(G)=N(T^+);$
  \item[$(3)$] $(I_Y+\delta TT^+)^{-1}\bar{T}$ maps $N(T)$ into $R(T);$
  \item[$(4)$] $(I_Y+\delta T T^+)^{-1}R(\bar{T})=R(T);$
  \item[$(5)$]  $(I_X+T^+\delta T)^{-1}N(T)=N(\bar T)$.
 \end{enumerate}
\noindent{\rm(B)} Further assume that $\delta T\in C(X,Y)$, $T^+\in
C(Y,X)$ and
\begin{equation}\label{3eqb}
c=\sup\{\|TT^+x\|\vert\,x\in D(T^+),\, \|x\|=1\}<+\infty,
\end{equation}
\rm{(e.g. $T$ satisfies conditions of Proposition \ref{2P1} (1))}.
If $bc<1$ \rm{(note that $c\ge 1$)}, then $G\in C(Y,X)$.
\end{theorem}
\begin{proof} We first prove statement (A).

$(1)\Rightarrow (2)$   We have $\bar T\in C(X,Y)$ and
$G=T^+(I_Y+\delta TT^+)^{-1}=(I_X+T^+\delta T)^{-1}T^+$ by Lemma
\ref{YL1}.

We now check that $\bar TG\bar T=\bar T$ on $D(T)$ and $G\bar TG=G$
on $D(T^+)$. We have
\begin{align*}
\bar{T}G\bar T&=(T+\delta T)T^+(I_Y+\delta TT^+)^{-1}(T+\delta T)\\
&=(T+\delta T)(I_X+T^+\delta T)^{-1}(T^+T+T^+\delta T)\\
&=(T+\delta T)(I_X+T^+\delta T)^{-1}(T^+T-I_X+I_X+T^+\delta T)\\
&=-\bar T(I_X+T^+\delta T)^{-1}(I_X-T^+T)+\bar T\\
&=\bar T
\end{align*}
on $D(T)$ by Lemma \ref{YL2} since $R(\bar{T})\cap N(T^+)=\{0\}$.
Also, we have
\begin{align*}
G\bar{T}Gy&=T^+(I_Y+\delta TT^+)^{-1}(T+\delta T)T^+(I_Y+\delta TT^+)^{-1}y\\
&=T^+(I_Y+\delta TT^+)^{-1}(I_Y+\delta TT^+)TT^+(I_Y+\delta TT^+)^{-1}y\\
&=T^+(I_Y+\delta TT^+)^{-1}y=Gy
\end{align*}
for any $y\in D(T^+)$.

From $G=T^+(I_Y+\delta TT^+)^{-1}=(I_X+T^+\delta T)^{-1}T^+$, we
obtain $R(G)=R(T^+)$ and $N(G)=N(T^+)$.

$(2)\Rightarrow (3)$ According to the proof of $(1)\Rightarrow (2)$,
we have
\begin{equation}\label{eq2g}
\bar T(I_X+T^+\delta T)^{-1}(I_X-T^+T)=0.
\end{equation}
Thus, by (\ref{eq2g}),
$$
(I_Y-TT^+)(I_Y+\delta TT^+)\bar T(I_X-T^+T)=(I_Y-TT^+)\delta
T(I_X+T^+\delta T)(I_X-T^+T)=0
$$
on $D(T)$. This means that $(I_Y+\delta TT^+)^{-1}\bar{T}$ maps
$N(T)$ into $R(T)$.

$(3)\Rightarrow (4)$ Let $x\in D(T)$ and put $x_1=T^+Tx$,
$x_2=(I_X-T^+T)x\in N(T)$. Then $(I_Y+\delta TT^+)^{-1}\bar Tx_2 \in
R(T)$ by the assumption. Since
$$
(I_Y+\delta TT^+)^{-1}\bar Tx_1=(I_Y+\delta TT^+)^{-1}(I_Y+\delta
TT^+)Tx_1=Tx_1\in R(T),
$$
it follows that $(I_Y+\delta TT^+)^{-1}R(\bar T)\subset R(T)$. On
the other hand, for any $x\in D(T)$
$$
(I_Y+\delta TT^+)Tx=\bar TT^+Tx\in R(\bar T)\subset D(T^+).
$$
So $R(T)\subset(I_Y+\delta TT^+)^{-1}R(\bar T)$ and consequently,
$R(T)=(I_Y+\delta TT^+)^{-1}R(\bar T)$.

$(4)\Rightarrow (1)$  Let $\xi\in R(\bar{T})\cap N(T^+)$. Then
$T^+\xi=0$ and $\xi=(I_Y+\delta TT^+)T\eta$ for some $\eta\in D(T)$.
Thus, $(I_X+T^+\delta T)T^+T\eta=0$ and hence $T^+T\eta=0$. This
implies that $\xi=0$.

The implication $(1)\Rightarrow (5)$ is Lemma \ref{YL2}. To complete
the proof, we now show the implication $(5)\Rightarrow (3)$. Since
$\bar T(I_X+T^+\delta T)^{-1}(I_X-T^+T)=0$, we have
\begin{align*}
T(I_X+T^+\delta T)^{-1}(I_X-T^+T)&=-(I_Y+\delta TT^+)^{-1}\delta T(I_X-T^+T)\\
&=-(I_Y+\delta TT^+)^{-1}\bar T(I_X-T^+T),
\end{align*}
that is, $(I_Y+\delta TT^+)^{-1}\bar T$ maps $N(T)$ into $R(T)$.

(B) To prove $G\in C(Y,X)$, let $\{y_n\}^\infty_{n=1}\subset D(T^+)$
and $y\in Y$, $x\in X$ such that $\|y_n-y\|\to 0$ and $\|Gy_n-x\|\to
0$ ($n\to\infty$). Set $z_n=(I_Y+\delta TT^+)^{-1}y_n\in D(T^+)$,
$n\ge 1$. Then $z_n=y_n-\delta TT^+z_n$, $n\ge 1$ and
$\|T^+z_n-x\|\to 0$ ($n\to\infty$). Since $\delta T$ is
$T$--bounded, we have, for any $m,\,n\ge 1$,
\begin{align*}
\|z_n-z_m\|&\le\|y_n-y_m\|+\|\delta TT^+(z_n-z_m)\|\\
&\le\|y_n-y_m\|+a\|T^+z_n-T^+z_m\|+b\|TT^+(z_n-z_m)\|\\
&\le\|y_n-y_m\|+a\|T^+z_n-T^+z_m\|+bc\|z_n-z_m\|.
\end{align*}
Thus, $\|z_n-z_m\|<(1-bc)^{-1}(\|y_n-y_m\|+a\|T^+z_n-T^+z_m\|)$,
$m,\,n\ge 1$ and that $\{z_n\}^\infty_{n=1}$ is a Cauchy sequence in
$Y$. Let $\|z_n-z\|\to 0$ as $n\to\infty$ for some $z\in Y$. Since
$T^+\in C(Y,X)$, it follows that $z\in D(T^+)$ and $T^+z=x$. From
$\delta TT^+z_n=y_n-z_n\stackrel{\|\cdot\|}{\longrightarrow}y-z$,
$T^+z_n \stackrel{\|\cdot\|}{\longrightarrow}x$ and $\delta T\in
C(X,Y)$, we get that $x\in D(\delta T)$ and $\delta Tx=y-z$. Thus
$y\in D(T^+)$, $x=T^+(y-\delta Tx)$ and hence $x=(I_X+T^+\delta
T)^{-1}T^+y=Gy$.
\end{proof}

\begin{remark}{\rm Let $T\in C(X,Y)$ such that $T^+\in C(Y,X)$ exists and let $\delta T\in B(X,Y)$ with $R(\delta T)
\subset D(T^+)$. In this case, we do not need Condition
(\ref{3eqb}). Put $\bar{T}=T+\delta T$. Then $\bar T\in C(X,Y)$ and
$T^+\delta T\in B(X,X)$ by Closed Graph Theorem. Assume that
$I_Y+\delta TT^+\colon D(T^+)\rightarrow D(T^+)$ is bijective and
$R(\bar T)\cap N(T^+)=\{0\}$. Then $G=(I_X+T^+\delta T)^{-1}T^+\in
C(Y,Y)$.

In fact, let $y\in Y$ and $x\in X$ and suppose that there is a
sequence $\{y_n\}$ in $Y$ such that $\|y_n-y\|\to 0$ and
$\|Gy_n-x\|\to 0$ ($n\to\infty$). Then
$$
T^+y_n=(I_X+T^+\delta T)(I_X+T^+\delta
T)^{-1}T^+y_n\stackrel{\|\cdot\|}{\longrightarrow}(I_X+T^+\delta
T)x.
$$
Since $T^+\in C(Y,X)$, we get that $y\in D(T^+)$ and
$T^+y=(I_X+T^+\delta T)x$. Consequently, $Gy=x$. Therefore, $\bar
T^+=T^+(I_Y+\delta TT^+)^{-1}\in C(Y,X)$ by Theorem \ref{DL2} (A).

In addition, if $T^+\in B(Y,X)$,  the results of Theorem \ref{DL2}
(A) are contained in \cite[Chapter 2]{X}. }
\end{remark}

\begin{remark}{\rm
Let $T\in C(X,Y)$ with $T^+\in B(Y,X)$ and let $\delta T\colon
D(\delta T)\rightarrow Y$ be a $T$--bounded linear operator with
$b<1$ and $D(T)\subset D(\delta T)$. Then $\delta T T^+\in B(Y,X)$.
Suppose that $I_Y+\delta T T^+$ is invertible in $B(Y,Y)$ and
$R(T+\delta T)\cap N(T^+)=\{0\}$. Then the bounded linear operator
$T^+(I_Y+\delta T T^+)^{-1}$ is a generalized inverse of $T+\delta
T$ by Theorem \ref{DL2}. This result is Theorem 2.1 of \cite{QHL}.
However, in this case, the equivalence of the conditions (1)---(5)
of Theorem \ref{DL2} (A) is not given in \cite{QHL}.

In addition, if there are constants $a,\,b>0$ such that
$$
a\|T^+\|+b\|TT^+\|<1,\ \|\delta Tx\|\le a\|x\|+b\|Tx\|,\
\forall\,x\in D(T),
$$
then $\|\delta TT^+\|<1$ and $b<1$ for $\|TT^+\|\ge 1$. Thus, $\bar
T$ is a closed operator and $I_Y+\delta TT^+$ is invertible in
$B(Y,Y)$. Therefore, the conditions (1)---(5) of Theorem \ref{DL2}
(A) are equivalent. This result is Theorem 2.1 in \cite{QH}. }
\end{remark}

Finally, combining Proposition \ref{2P1} (2) with Theorem \ref{DL2}
(A), we have
\begin{corollary}
Let $T\in C(X,Y)$ with $R(T)$ closed such that $T^+_{P,\,Q}$ exists.
Let $\delta T\in B(X,Y)$ such that $I_X+T^+_{P,\,Q}\delta T$ is
invertible in $B(X,X)$ and $R(T+\delta T)\cap N(T^+_{P,\,Q})=\{0\}$.
Then $R(T+\delta T)$ is closed and $(T+\delta T)^+_{\bar P,\,\bar
Q}=(I_X+T^+_{P,\,Q}\delta T)^{-1}T^+_{P,\,Q}$, where $\bar
P=(I_X+T^+_{P,\,Q}\delta T)^{-1}P(I_X+T^+_{P,\,Q}\delta T)$ and
$\bar Q=(I_Y+\delta TT^+_{P,\,Q}) T^+_{P,\,Q}(I_Y+\delta
TT^+_{P,\,Q})^{-1}$.
\end{corollary}

\end{document}